\documentclass[letter,11pt,reqno]{amsart}

\usepackage[utf8]{inputenc}

\usepackage{etex}

\usepackage{xcolor}

\definecolor{verydarkblue}{rgb}{0,0,0.5}

\usepackage[
    breaklinks,
    colorlinks,
    citecolor=verydarkblue,
    linkcolor=verydarkblue,
    urlcolor=verydarkblue,
    pagebackref=true,
    hyperindex
]{hyperref}

\backrefenglish

\usepackage{fancyhdr}

\usepackage[
    hscale=0.70,
    vscale=0.775,
    headheight=13pt,
    centering,
]{geometry}

\usepackage{amsmath}
\usepackage{amsthm}
\usepackage{amssymb}
\usepackage{mathtools}
\usepackage{stmaryrd}
\usepackage{mathdots}
\usepackage{framed}
\usepackage[capitalize]{cleveref}
\usepackage{array}
\usepackage[alphabetic]{amsrefs}
\usepackage[all,cmtip]{xy}

\theoremstyle{plain}

\crefname{introtheorem}{Theorem}{Theorems}

\newtheorem{theorem}{Theorem}[section]
\newtheorem{proposition}[theorem]{Proposition}
\newtheorem{lemma}[theorem]{Lemma}
\newtheorem{corollary}[theorem]{Corollary}

\theoremstyle{definition}

\newtheorem{example}[theorem]{Example}

\numberwithin{figure}{section}

\numberwithin{equation}{section}

\IfFileExists{./article-style.tex}{

\pagestyle{fancy}
\fancyhead{}
\fancyfoot{}
\fancyhead[LE,RO]{\small \thepage}
\fancyhead[RE]{\small \nouppercase{\rightmark}}
\fancyhead[LO]{\small \nouppercase{\leftmark}}

\setcounter{tocdepth}{1}

\makeatletter

\global\BR@BackrefAlttrue

\renewcommand*{\backrefalt}[4]{%
    \tiny%
    (%
    \ifcase #1 not cited%
          \or cit.~on~p.~#2%
          \else cit.~on~pp.~#2%
    \fi%
    )%
}

\def\print@backrefs#1{%
    \space\SentenceSpace%
    \begingroup%
        \expandafter\providecommand\csname brc@#1\endcsname{0}%
        \expandafter\providecommand\csname brcd@#1\endcsname{0}%
        \expandafter\backrefalt%
            \csname brc@#1\expandafter\endcsname%
            \csname brl@#1\expandafter\endcsname%
            \csname brcd@#1\expandafter\endcsname%
            \csname brld@#1\endcsname%
    \endgroup%
}

\def\maketitle{\par
  \@topnum\z@ 
  \@setcopyright
  \thispagestyle{empty}
  \ifx\@empty\shortauthors \let\shortauthors\shorttitle
  \else \andify\shortauthors
  \fi
  \@maketitle@hook
  \begingroup
  \@maketitle
  \toks@\@xp{\shortauthors}\@temptokena\@xp{\shorttitle}%
  \toks4{\def\\{ \ignorespaces}}
  \edef\@tempa{%
    \@nx\markboth{\the\toks4
      \@nx\MakeUppercase{\the\toks@}}{\the\@temptokena}}%
  \@tempa
  \endgroup
  \c@footnote\z@
    \renewcommand{\footnoterule}{%
      \kern -3pt
      \hrule width \textwidth height .5pt
      \kern 2pt
    }
  {
    \renewcommand\thefootnote{}
    \vspace{-2em}
    \footnote{
      \par\vspace{-1.2em}\noindent%
      \setlength{\parindent}{0pt}%
      \def\@footnotetext##1{\noindent{\footnotesize##1}\par}%
      \let\@makefnmark\relax  \let\@thefnmark\relax
      \ifx\@empty\@date\else \@footnotetext{\@setdate}\fi
      \ifx\@empty\@subjclass\else \@footnotetext{\@setsubjclass}\fi
      \ifx\@empty\@keywords\else \@footnotetext{\@setkeywords}\fi
      \ifx\@empty\thankses\else \@footnotetext{%
        \@setthanks}%
      \fi
    }
    \addtocounter{footnote}{-1}
  }
  \@cleartopmattertags
}

\def\@adminfootnotes{\@empty}

\def\@settitle{\begin{center}%
  \baselineskip14\p@\relax
    \bfseries
\Large
  \@title
  \end{center}%
}

\def\@setauthors{%
  \begingroup
  \def\thanks{\protect\thanks@warning}%
  \trivlist
  \centering\footnotesize \@topsep30\p@\relax
  \advance\@topsep by -\baselineskip
  \item\relax
  \author@andify\authors
  \def\\{\protect\linebreak}%
  \large{\authors}%
  \ifx\@empty\contribs
  \else
    ,\penalty-3 \space \@setcontribs
    \@closetoccontribs
  \fi
  \endtrivlist
  \endgroup
}

\def\@setaddresses{\par
  \nobreak \begingroup
\footnotesize
  \def\author##1{\end{minipage}\hskip 1sp \begin{minipage}{.5\textwidth}\raggedright%
    ~\\[2em]{\bf##1}\\[.5em]%
  }%
  \interlinepenalty\@M
  \def\address##1##2{\begingroup
    {\ignorespaces##2}\endgroup\\[.5em]}%
  \def\curraddr##1##2{\begingroup
    \@ifnotempty{##2}{\nobreak\indent\curraddrname
      \@ifnotempty{##1}{, \ignorespaces##1\unskip}\/:\space
      ##2\par}\endgroup}%
  \def\email##1##2{\begingroup
    \@ifnotempty{##2}{\nobreak\indent
      \@ifnotempty{##1}{, \ignorespaces##1\unskip}
      \ttfamily##2\par}\endgroup}%
  \def\urladdr##1##2{\begingroup
    \def~{\char`\~}%
    \@ifnotempty{##2}{\nobreak\indent\urladdrname
      \@ifnotempty{##1}{, \ignorespaces##1\unskip}\/:\space
      \ttfamily##2\par}\endgroup}%
  \setlength{\parindent}{0pt}%
  \vfill%
  {
  \begin{minipage}{0mm}
  \addresses
  \end{minipage}
  }
  \endgroup
}

\renewcommand{\author}[2][]{%
  \ifx\@empty\authors
    \gdef\authors{#2}%
    \g@addto@macro\addresses{\author{#2}}%
  \else
    \g@addto@macro\authors{\and#2}%
    \g@addto@macro\addresses{\author{#2}}%
  \fi
  \@ifnotempty{#1}{%
    \ifx\@empty\shortauthors
      \gdef\shortauthors{#1}%
    \else
      \g@addto@macro\shortauthors{\and#1}%
    \fi
  }%
}
\edef\author{\@nx\@dblarg
  \@xp\@nx\csname\string\author\endcsname}

\def\@secnumfont{\@empty}

\def\section{\@startsection{section}{1}%
  \z@{.7\linespacing\@plus\linespacing}{.5\linespacing}%
  {\large\bfseries\centering}}

\makeatother

}{}

\usepackage{marginnote}

\def\cF{\mathcal{F}}

\def\cK{\mathcal{K}}

\def\cQ{\mathcal{Q}}
\def\cR{\mathcal{R}}

\def\I{\mathcal{I}}

\def\O{\mathcal{O}}

\def\fa{\mathfrak{a}}

\def\fm{\mathfrak{m}}

\def\a{\alpha}
\def\b{\beta}
\def\g{\gamma}

\def\n{\nu}
\def\m{\mu}

\def\p{\pi}
\def\r{\rho}

\def\t{\tau}

\def\D{\Delta}

\def\Om{\Omega}

\def\.{\cdot}
\def\^{\widehat}
\def\~{\widetilde}
\def\o{\circ}

\def\surj{\twoheadrightarrow}

\def\({\left(}
\def\){\right)}

\def\*{{}^*}

\renewcommand{\and}{ \ \ \text{ and } \ \ }

\DeclareMathOperator{\im} {Im}

\DeclareMathOperator{\Gr} {Gr}

\DeclareMathOperator{\Spec} {Spec}

\DeclareMathOperator{\id} {id}

\begin{document}

\title{Nash blow-ups of jet schemes}

\author{Tommaso de Fernex}

\address[T.\ de Fernex]{%
    Department of Mathematics\\
    University of Utah\\
    155 South 1400 East\\
    Salt Lake City, UT 48112\\
    USA%
}

\email{defernex@math.utah.edu}

\author{Roi Docampo}

\address[Roi Docampo]{%
    Department of Mathematics\\
    University of Oklahoma\\
    601 Elm Avenue, Room 423\\
    Norman, OK 73019\\
    USA%
}

\email{roi@ou.edu}


\thanks{The research of the first author was partially supported by NSF Grant DMS-1700769.}

\begin{abstract}
Given an arbitrary projective birational morphism of varieties, we provide a
natural and explicit way of constructing relative compactifications of the maps
induced on the main components of the jet schemes. In the case the morphism is
the Nash blow-up of a variety, such relative compactifications are shown to be
given by the Nash blow-ups of the main components of the jet schemes.
\end{abstract}

\maketitle

\section{Introduction}

The \emph{Nash blow-up} of a variety is defined as the universal
projective birational morphism for which the pull-back of the sheaf of
differentials admits a locally free quotient of the same rank.
The name comes from John Nash, who is generally credited for having promoted
the question of whether singularities of algebraic varieties can always be
resolved by finitely many iterations of such blow-ups; before him, the question
had already been considered by Semple \cite{Sem54}. The property is known to
hold for curves of characteristic zero, and to fail in positive characteristics
\cite{Nob75}. A variant of this question, where Nash blow-ups are alternated
with normalizations, has been settled affirmatively for surfaces of
characteristic zero by Spivakovsky \cite{Spi90}, building on \cite{Hir83}.
Higher order Nash blow-ups have been defined and studied by Yasuda
\cite{Yas07}.

The Nash blow-up can be thought as the universal operation separating multiple
limits of tangent spaces, and hence its construction relates to the geometry of
the main component of the first jet scheme of the variety. It is however
unclear a priori how the Nash blow-up of a variety should relate to the Nash
blow-up of such component. Even less obvious is whether there should be a
relationship with the Nash blow-ups of the main components of the higher jet
schemes of the variety. 

The following result shows that these Nash blow-ups are not just related, but
in fact they essentially determine each other. 

\begin{theorem}
\label{th:jet-Nash-blowup}
Let $X$ be a variety. For every $n$, the main component of the $n$-th jet
scheme of the Nash blow-up of $X$ has an open immersion into the Nash blow-up
of the main component of the $n$-th jet scheme of $X$, and such immersion is
compatible with the respective natural map to the $n$-th jet scheme of $X$.
\end{theorem}

Denoting by $N(X) \to X$ the Nash blow-up of a variety and by $J_n'(X)$ the
main component of the $n$-th jet scheme of $X$, \cref{th:jet-Nash-blowup} can
be rephrased by saying that the Nash blow-up $N(J_n'(X)) \to J_n'(X)$ gives a
relative compactification of the map $J_n'(N(X)) \to J_n'(X)$ induced on
$n$-jets by the Nash blow-up of $X$.
This implies that the Nash blow-up of a variety $X$ can equivalently be
characterized as the universal projective birational morphism $Y \to X$ such
that, for every $n$, the pull-back of $\Om_{J_n'(X)}$ via $J_n'(Y) \to J_n'(X)$
has a locally free quotient of the same rank.
The theorem also implies that the Nash blow-up of $J_n'(X)$ induces the Nash
blow-up of $X$ under the natural section (the `zero section') of the projection
$J_n'(X) \to X$. It was shown by Ishii \cite{Ish09} that if a variety $X$ is
singular then all of its jet schemes are singular, and
\cref{th:jet-Nash-blowup} implies that, if the ground field is algebraically
closed of characteristic zero, then in fact the main components of the jet
schemes are already singular.

The proof of \cref{th:jet-Nash-blowup} uses the description of the sheaves of
differentials on jet schemes given in \cite{dFD} in combination with
\cref{th:jet-Nash-transform} (stated below), which addresses a related question
in a more general context.

Suppose that $\m \colon Y \to X$ is an arbitrary projective birational morphism
of varieties. By functoriality, $\m$ induces for every $n$ a morphism on jet
schemes $\m_n \colon J_n(Y) \to J_n(X)$, and hence, by restriction, a
birational morphism $\m_n' \colon J_n'(Y) \to J_n'(X)$ between the main
components of the jet schemes. In general, $\m_n'$ is not a projective
morphism, and one can ask whether there are natural ways of constructing
relative compactifications of $\m_n'$. The next theorem provides an answer to
this question. 

The morphism $\m$ can be described as the blow-up of an ideal sheaf $\I \subset
\O_X$, and a way to approach the question is to look for natural ways of
constructing an ideal sheaf $\fa_n \subset \O_{J_n'(X)}$ whose blow-up gives a
relative compactification of $\m_n'$. Doing this directly seems hard: while a
posteriori we will provide an explicit formula for computing the local
generators of such an ideal $\fa_n$ in terms of the generators of $\I$, the
formula will show that the complexity of $\fa_n$ grows fast even in simple
examples, an indication that looking at ideals might not be the best approach.

Instead, we view $\m$ as the \emph{Nash transformation} $N(\cF) \to X$ of a
coherent sheaf $\cF$, as defined for instance in \cite{OZ91}. In this language,
the blow-up of an ideal $\I \subset \O_X$ is the same as the Nash
transformation $N(\I) \to X$ of the ideal, and the Nash blow-up of a variety
$X$ is defined to be the Nash transformation $N(\Om_X) \to X$ of the sheaf of
differentials of $X$. In general, the Nash transformation of a coherent sheaf
$\cF$ of rank $r$ is defined using the Grassmann bundle of locally free
quotients of rank $r$ of $\cF$, and is a projective birational morphism.
Conversely, every projective birational morphism $\m \colon Y \to X$ can be
realized as a Nash transformation of some coherent sheaf $\cF$ on $X$. 
\begin{theorem}
\label{th:jet-Nash-transform}
Let $X$ be a variety over a field $k$, and let $\m \colon N(\cF) \to X$ be the
Nash transformation of a coherent sheaf $\cF$ on $X$. For every $n$, let 
\[
\xymatrix{
J_n'(X) \times \Delta_n \ar[d]_{\r_n'} \ar[r]^(.65){\g_n'} & X \\
J_n'(X)
&
}
\]
be the diagram induced by restriction from the universal $n$-jet of $X$; here,
we denote $\D_n = \Spec k[t]/(t^{n+1})$. Define 
\[
\cF_n' := (\r_n')_*(\g_n')^*\cF.
\]
Then the induced map $\m_n' \colon J_n'(N(\cF)) \to J_n'(X)$ factors as
\[
\xymatrix@C=20pt{
J_n'(N(\cF)) \ar@{^(->}[r]^(.55){\iota_n}
& N(\cF_n') \ar[r]^{\n_n} 
& J_n'(X)
}
\]
where $\iota_n$ is an open immersion and $\n_n$ is the Nash transformation of
$\cF_n'$. 
\end{theorem}

If in this theorem we take $\cF = \I \subset \O_X$, an ideal sheaf on $X$, then
$\cF_n'$ is not an ideal sheaf. However, the sheaf $\wedge^{(n+1)}\cF_n'$,
modulo torsion, is isomorphic to an ideal sheaf $\fa_n$, and $N(\cF_n') =
N(\fa_n)$. Our approach enables us to make explicit computations and hence to
provide a formula for the generators of $\fa_n$. 

A motivation for \cref{th:jet-Nash-transform} comes from the Nash problem on
families of arcs through the singularities of a variety \cite{Nas95} and, more
specifically, from the problem of lifting wedges \cite{LJ80,Reg06}. In
dimension two, the Nash problem has been settled in characteristic zero in
\cite{FdBPP12} but it remains open in positive characteristics. The algebraic
proof given in \cite{dFD16} may be adaptable to positive characteristics,
provided one can avoid certain wild ramifications that could occur in the
proof. A possible approach is to look for suitable
deformations of wedges, and this requires working with relative
compactifications of the maps $J_n(Y) \to J_n(X)$ where $Y \to X$ is the
minimal resolution of the surface. \cref{th:jet-Nash-transform} provides a
first step in this direction. 

\subsection{Acknowledgments} 

We thank Mircea Musta\c t\u a for pointing out an error in a previous version of the paper
and the referee for a careful reading of the paper and valuable comments and corrections.

\section{Proofs}

We work over an arbitrary field $k$. For every integer $n \ge 0$, the
\emph{$n$-th jet scheme} $J_n(X)$ of a scheme $X$ is the scheme representing
the functor of points defined by
\[
J_n(X)(Z) = X(Z \times \D_n)
\]
for any scheme $Z$, where $\D_n = \Spec k[t]/(t^{n+1})$. We denote by 
\[
\xymatrix{
J_n(X) \times \Delta_n \ar[d]_{\r_n} \ar[r]^(.65){\g_n} & X \\
J_n(X)
&
}
\]
the universal $n$-jet of $X$. For generalities about jet schemes, we refer
to \cite{Voj07,EM09}. If $X$ is a variety, then there exists a unique
irreducible component of $J_n(X)$ dominating $X$, and this component has
dimension $(n+1)\dim X$. We shall denote it by $J_n'(X)$ and call it the
\emph{main component} of $J_n(X)$.

Given a coherent sheaf $\cF$ on a scheme $X$, and a positive integer $r$, we
denote by $\Gr(\cF,r)$ the \emph{Grassmann bundle} over $X$ parameterizing
locally free quotients of $\cF$ of rank~$r$, where by the term \emph{quotient}
we mean an equivalence class of surjective maps from the same source where two surjections
are identified whenever they have the same kernel. 
This scheme represents the functor of points given by 
\begin{align*}
\Gr(\cF,r)(Z) &= 
\big\{ \big(Z \xrightarrow{p} X,\, p^*\cF \twoheadrightarrow {\mathcal Q}\big) 
\mid \\
& \hskip1cm \text{${\mathcal Q}$ locally free sheaf on $Z$ of rank $r$}\big\}
\end{align*}
for any scheme $Z$.

Suppose now that $X$ is a variety, and let $\cF$ be a coherent sheaf on $X$ of
rank $r$. The \emph{Nash transformation} of $\cF$ is defined to be the
irreducible component of $\Gr(\cF,r)$ dominating $X$, and is denoted by
$N(\cF)$. The natural projection $\Gr(\cF,r) \to X$ induces the blow-up map
$N(\cF) \to X$. The \emph{Nash blow-up} $N(X) \to X$ is, by definition, the
Nash transformation of the sheaf of K\"ahler differentials $\Om_X$.

\begin{proof}[Proof of \cref{th:jet-Nash-transform}]
The sheaf $\cF_n'$ is the restriction, under the inclusion $J_n'(X) \subset J_n(X)$, of
the sheaf
\[
\cF_n := (\r_n)_*\g_n^*\cF.
\]
By construction, $J_n'(N(\cF))$ is an irreducible component of the jet scheme
$J_n(\Gr(\cF,r))$. Similarly, observing that $\cF_n'$ is a sheaf of rank
$(n+1)r$ and keeping in mind that $J_n'(X)$ is an irreducible component of
$J_n(X)$, we see that $N(\cF_n')$ is an irreducible component of
$\Gr(\cF_n,(n+1)r)$. We claim that there is a universally injective map
\[
i \colon J_n(\Gr(\cF,r)) \hookrightarrow \Gr(\cF_n,(n+1)r),
\]
defined over $X$, which agrees with the natural identification of these schemes
over the open set where $X$ is smooth and $\cF$ is locally free, 
and restricts to an open immersion from $J_n'(N(X))$ to $N(J_n'(X))$. 
Note that the existence of such a map implies the statement of the theorem. 

In order to prove this claim, we compare the functors of points of the schemes
$J_n(\Gr(\cF,r))$ and $\Gr(\cF_n,(n+1)r)$. For every scheme $Z$, we have
\begin{align*}
J_n(\Gr(\cF,r))(Z) &= \Gr(\cF,r)(Z \times \D_n) \\
&= \big\{ \big(Z \times \Delta_n \xrightarrow{\alpha} X,\, 
\alpha^*\cF \twoheadrightarrow {\mathcal Q}\big) 
\mid \\
& \hskip1cm \text{${\mathcal Q}$ locally free sheaf on $Z \times \Delta_n$ of rank $r$}\big\}
\end{align*}
and
\begin{align*}
\Gr(\cF_n,(n+1)r)(Z) 
&= \big\{ \big(Z \xrightarrow{\beta} J_n(X),\, 
\beta^*\cF_n \twoheadrightarrow {\mathcal R}\big) 
\mid \\
&\hskip1cm \text{${\mathcal R}$ locally free sheaf on $Z$ of rank $(n+1)r$}\big\}.
\end{align*}

By the description of $J_n(X)$ via the functor of points, every $\beta \colon Z
\to J_n(X)$ corresponds to a unique $\alpha \colon Z \times \Delta_n \to X$,
and for any such pair of maps there is a commutative diagram
\[
\xymatrix@C=40pt{
Z \times \Delta_n \ar[d]_{\p} \ar[r]_(.45){\b\times\id_{\D_n}} \ar@/^20pt/[rr]^\alpha
& J_n(X) \times \Delta_n \ar[d]^{\r_n} \ar[r]_(.65){\g_n}
& X \\
Z \ar[r]_(.45)\beta 
& J_n(X)
&
}
\]
where $\p$ is the projection onto the first factor. Note that taking
push-forward along $\p$ of a sheaf on $Z \times \D_n$ simply means that we are
restricting scalars to $\O_Z$ and forgetting the given $\O_{Z \times
\D_n}$-module structure of the sheaf. 

By the definition of $\cF_n$ and base-change, which holds in this setting 
because $\r_n$ and $\p$ are affine, we have
\[
\b^*\cF_n = \p_*\a^*\cF.
\]
Using the identification $J_n(X)(Z) = X(Z \times \D_n)$, the above formula
yields the following alternative description of the functor of points of
$\Gr(\cF_n,(n+1)r)$:
\begin{align*}
\Gr(\cF_n,(n+1)r)(Z) 
&= \big\{ 
\big(
Z\times\D_n \xrightarrow{\a} X,\,
\p_*\a^*\cF \twoheadrightarrow {\mathcal R}
\big) 
\mid \\
& \hskip1cm \text{${\mathcal R}$ locally free sheaf on $Z$ of rank $(n+1)r$}\big\}.
\end{align*}

For every locally free sheaf ${\mathcal Q}$ on $Z \times \Delta_n$ of
rank $r$, the push-forward $\p_*{\mathcal Q}$ is a locally free sheaf on $Z$ of
rank $(n+1)r$. Taking push-forwards via $\p$ is exact,
and any two quotients of $\a^*\cF$ are identified (i.e., they define the same kernel in $\a^*\cF$)
if and only if their push-forwards are identified as quotients of $\p_*\a^*\cF$
(i.e., they define the same kernel in $\p_*\a^*\cF$).
This means that taking push-forwards via $\p$ defines a natural injection 
\[
J_n(\Gr(\cF,r))(Z) \hookrightarrow \Gr(\cF_n,(n+1)r)(Z).
\]
As this holds for every scheme $Z$, we deduce that there is a
naturally defined universally injective morphism
\[
i \colon J_n(\Gr(\cF,r)) \hookrightarrow \Gr(\cF_n,(n+1)r).
\]
It is immediate to see that $i$ is defined over $X$ and therefore it agrees with the natural
identification of these schemes over the open set where $X$ is smooth and 
$\cF$ is locally free. Furthermore, 
the restriction of $i$ to $J_n'(N(\cF))$ gives a universally injective map
$\iota_n \colon J_n'(N(\cF)) \to N(\cF_n')$. To finish the proof, we need to show that $\iota_n$ 
is a local isomorphism, that is, it induces isomorphisms on all local rings. To
this end, we prove the following property. 

\begin{lemma}
\label{th:lifting}
Let $(A,\fm)$ be a Noetherian local domain over $k$, and set $U = \Spec A$ and $P =
\Spec A/\fm$. Assume that $f_0$ and $g$ are morphisms as in the diagram
\[
\xymatrix{
P \ar[r]^(.33){f_0} \ar@{^(->}[d] & J_n(\Gr(\cF,r)) \ar@{^(->}[d]^i \\
U \ar[r]_(.25){g} \ar@{..>}[ru]^(.38){f} & \Gr(\cF_n,(n+1)r)
}
\]
such that the square sub-diagram commutes and the image of $g$ is a dense subset of $N(\cF'_n)$.
Then there exists a unique morphism $f$ (marked by the dotted arrow in
the diagram) making the whole diagram commute.
\end{lemma}

\begin{proof}
Suppose $f_0$ and $g$ are given. Let $\p_0 \colon P \times\D_n \to P$ and $\p
\colon U \times \D_n \to U$ denote the respective projections to the first
components. By the descriptions of the functors of points, we can write 
\[
f_0 = \big(P \times \D_n \xrightarrow{\a_0} X,\, \a_0^*\cF \surj \cQ\big),
\]
where $\cQ$ is a locally free $A[t]/(t^{n+1})$-module of rank $r$, and
\[
g = \big(U\times\D_n \xrightarrow{\a} X,\, \p_*\a^*\cF \surj \cR\big)
\] 
where $\cR$ is a locally free $A$-module of rank $(n+1)r$. The commutativity of
the square sub-diagram in the statement means that $\a_0$ is the restriction of
$\a$ and $\cR \otimes_A A/\fm = (\p_0)_*\cQ$. The fact that the image of $g$ 
is dense in $N(\cF'_n)$ implies that $\a$ is dominant, and hence 
$\p_*\a^*\cF$ is a sheaf of rank $(n+1)r$. Since 
$\cR$ is a locally free quotient of the same rank of $\p_*\a^*\cF$, the kernel $\cK$
of $\p_*\a^*\cF \to \cR$ is the torsion $A$-submodule of $\p_*\a^*\cF$.
Every element of $\sum_{i \ge 0} t^i\cK$, viewed as an $A$-submodule of $\p_*\a^*\cF$,
is torsion, and therefore we have $\sum_{i \ge 0} t^i\cK = \cK$. This shows that $\cK$ is 
an $A[t]/(t^{n+1})$-submodule of $\p_*\a^*\cF$ and hence $\cR$ is an 
$A[t]/(t^{n+1})$-module quotient of $\p_*\a^*\cF$.
This gives the lift $f$ of $g$ as in the diagram, which is clearly unique
and makes the diagram commute.
\end{proof}

We apply \cref{th:lifting} to the local rings of $N(\cF_n')$ at the
points in the image of $\iota_n$. Using the fact that $i$ is injective on the
functors of points, we deduce that $i$ induces isomorphisms on the local rings.

To see this last implication, let $\O_q$ denote the local ring of
$J_n'(N(\cF))$ at a point $q$, and let $\O_p$ denote the local ring of
$N(\cF_n')$ at $p = \iota_n(q)$.  Let $g \colon \Spec \O_p \to
\Gr(\cF_n,(n+1)r)$ and $h \colon \Spec \O_q \to J_n(\Gr(\cF,r))$ be the natural
maps, and let $j \colon \Spec \O_q \to \Spec \O_p$ be the map induced by $\iota_n$. 

We have the following diagram:
\[
\xymatrix{
\Spec\O_q \ar@{^(->}[r]^(.45)h \ar[d]^{j} & J_n(\Gr(\cF,r)) \ar[d]^i \\ 
\Spec\O_p \ar@{^(->}[r]_(.37)g \ar@{..>}@/^15pt/[u]^s \ar@{..>}[ur]^(.45)f 
& \Gr(\cF_n,(n+1)r).
}
\]
Here, the square sub-diagram is commutative, $f$ exists by \cref{th:lifting}
and hence satisfies 
\begin{equation}
\label{eq:if=g}
i \o f = g,
\end{equation} 
and the universal property of local rings implies that $f$ factors through $h$,
so that we have a morphism $s$, as in the diagram, satisfying
\begin{equation}
\label{eq:hs=f}
h \o s = f.
\end{equation} 
Using the commutativity of the square sub-diagram and \cref{eq:if=g}, we get
\[
i\o h = g \o j = i \o f \o j.
\]
Then, using the fact that $i$ is injective at the level of functors of points
and hence is a monomorphism, we deduce that
\begin{equation}
\label{eq:h=fi_q}
h = f \o j.
\end{equation}

Now, using \cref{eq:hs=f,eq:h=fi_q}, we get
\[
h = f \o j = h \o s \o j,
\]
and since $h$ is a monomorphism, this implies that $s \o j$ is the identity
of $\Spec \O_q$. Using \cref{eq:hs=f,eq:h=fi_q} in a different order, we get
\[
f = h\o s = f \o j \o s.
\]
Since $g$ is a monomorphism, it follows by \cref{eq:if=g} that $f$ is a
monomorphism, and this implies that $j \o s$ is the identity of $\Spec \O_p$.
This proves that $j$ is an isomorphism, which completes the proof of the
theorem.
\end{proof}

\begin{proof}[Proof of \cref{th:jet-Nash-blowup}]
By \cite[Theorem~B]{dFD}, there is an isomorphism
\[
\Om_{J_n(X)} \cong (\r_n)_*\g_n^*\Om_X,
\]
and this implies that
\[
N(J_n'(X)) = N\big((\r_n)_*\g_n^*\Om_X \otimes_{\O_{J_n(X)}}\O_{J_n'(X)}\big)
= N\big((\r_n')_*(\g_n')^*\Om_X\big),
\]
where $\r_n'$ and $\g_n'$ are the restrictions of $\r_n$ and $\g_n$ to $J_n'(X)
\times \D$. Therefore \cref{th:jet-Nash-blowup} reduces to
\cref{th:jet-Nash-transform} with $\cF = \Om_X$. 
\end{proof}

\begin{corollary}
\label{th:cor-jet-Nash-blowup}
For any variety $X$, the following properties are equivalent:
\begin{enumerate}
\item
the Nash blow-up $N(J_n'(X)) \to J_n'(X)$ is an isomorphism for some $n\ge 0$;
\item
the Nash blow-up $N(J_n'(X)) \to J_n'(X)$ is an isomorphism for every $n\ge 0$.
\end{enumerate}
\end{corollary}

\begin{proof}
By \cref{th:jet-Nash-blowup}, both properties are equivalent to the fact that
the Nash blow-up $N(X) \to X$ is an isomorphism.
\end{proof}

In positive characteristics, there are examples of singular varieties whose
Nash blow-up is an isomorphism (see \cite[Example~1]{Nob75}), and
\cref{th:cor-jet-Nash-blowup} implies that this property, whenever it holds,
propagates through all the jet schemes, and conversely.  

By contrast, when the ground field is algebraically closed of characteristic
zero the Nash blow-up is an isomorphism if and only if the variety is smooth
(see \cite[Theorem~2]{Nob75}). It is elementary to show that the jet schemes of
a smooth variety are smooth, and conversely it was proved in \cite{Ish09} that
if $X$ is a singular variety then all its jet schemes $J_n(X)$ are singular.
With the above assumptions on the ground field, we deduce the following
stronger statement from \cref{th:cor-jet-Nash-blowup}.

\begin{corollary}
If $X$ is a singular variety defined over an algebraically closed field of
characteristic zero, then the main component of $J_n'(X)$ of $J_n(X)$ is
singular for every $n$.
\end{corollary}

\section{Computational aspects}

After viewing a projective birational morphism $\m \colon Y \to X$ as the Nash
transformation of a coherent sheaf $\cF$ on a variety $X$,
\cref{th:jet-Nash-transform} provides a construction of a relative
compactification of the induced map $\m_n' \colon J_n'(Y) \to J_n'(X)$ by
taking the Nash transformation of an explicitly described sheaf $\cF_n'$ on
$J_n'(X)$. Such transformation is a projective birational morphism, and
therefore can also be described as the blow-up of an ideal sheaf $\fa_n$ on
$J_n'(X)$. In this section we explain how to compute such ideal.

For simplicity, we assume that $X = \Spec R$ is affine. The following diagram
provides the algebraic counterpart of the restriction to $J_n'(X)$ of the
universal $n$-jet:
\[
    \xymatrix{
        R_n'[t]/(t^{n+1})
        &
        R \ar[l]_-{(\gamma_n')^\sharp}
        \\
        R_n' \ar[u]_-{(\rho_n')^\sharp}
    }
\]
Here $R_n'$ is a quotient of $R_n$, the algebra of Hasse--Schmidt differentials
of order at most $n$, $(\rho_n')^\sharp$ is the natural inclusion map, and
$(\gamma_n')^\sharp$ is induced by the homomorphism
\[
\g_n^\sharp \colon R \to R_n[t]/(t^{n+1}), \quad
    f \mapsto \sum_{i=0}^n D_i(f)\, t^i,
\]
where $(D_0, D_1, \ldots, D_n)$ is the universal Hasse--Schmidt derivation of
order $n$. With this notation, we have $J_n(X) = \Spec R_n$ and $J_n'(X) =
\Spec R_n'$.

If $\t(\cF)$ denotes the torsion of $\cF$, then the two sheaves $\cF_n'$ and
$(\cF/\t(\cF))_n'$ have the same torsion free quotient. We can therefore assume
without loss of generality that $\cF$ is torsion free. Let $F$ denote the
$R$-module associated to $\cF$. If $r$ is the rank of $F$, we can then realize
$F$ as a submodule of $R^r$. Picking a set of generators for $F$ of cardinality
$s$, we obtain a matrix $M \in {\rm Mat}_{r\times s}(R)$ such that $F = \im M$.
Notice that, to produce an ideal whose blow-up gives $Y \to X$, one can take
the ideal generated by the $r \times r$ minors of $M$.

The relative compactification of $\m_n' \colon J_n'(Y) \to J_n'(X)$ constructed
in \cref{th:jet-Nash-transform} is given by the Nash transformation of the
$R_n'$-module
\[
    F_n' := (\gamma_n')^\sharp(F) \cdot (R_n'[t]/(t^{n+1}))^r,
\]
where the $R_n'$-module structure is defined via $(\rho_n')^\sharp$. A
straightforward computation shows that $F_n = \im M_n$ where $M_n \in {\rm
Mat}_{(n+1)r \times (n+1)s}(R_n')$ is the matrix given in block form by
\[
    M_n =
    \begin{bmatrix}
        D_0(M) & 0      & \cdots & 0      \\
        D_1(M) & D_0(M) & \cdots & 0      \\
        \vdots & \vdots & \ddots & \vdots \\
        D_n(M) & D_{n-1}(M) & \cdots & D_0(M) \\
    \end{bmatrix}.
\]
Here $D_i(M)$ is the matrix obtained from $M$ by applying $D_i$ to each entry.
By construction, we have the following property.

\begin{proposition}
With the above notation, the morphism $N(F_n') \to J_n'(X)$ is the blow-up of
the ideal $\fa_n \subset R_n'$ generated by the $(n+1)r \times (n+1)r$ minors
of $M_n$.
\end{proposition}

The next example shows the computation of the first few ideals $\fa_n$ in a
simple case.

\begin{example}
\label{example1}
Consider $X = \mathbb A^2 = \Spec k[x,y]$, and let $Y \to X$ be the blow-up of
the maximal ideal $(x,y)$. Taking $F$ to be the maximal ideal, we have 
\begin{align*}
M_0 &= \begin{bmatrix} x & y \end{bmatrix},\\
M_1 &=    \begin{bmatrix}
       x   & y   & 0   & 0 \\
       x_1 & y_1 & x   & y \\
    \end{bmatrix},\\
M_2 &=    \begin{bmatrix}
       x   & y   & 0   & 0   & 0   & 0 \\
       x_1 & y_1 & x   & y   & 0   & 0 \\
       x_2 & y_2 & x_1 & y_1 & x   & y \\
    \end{bmatrix}, \\
M_3 &=    \begin{bmatrix}
       x   & y   & 0   & 0   & 0   & 0   & 0 & 0 \\
       x_1 & y_1 & x   & y   & 0   & 0   & 0 & 0 \\
       x_2 & y_2 & x_1 & y_1 & x   & y   & 0 & 0 \\
       x_3 & y_3 & x_2 & y_2 & x_1 & y_1 & x & y \\
    \end{bmatrix},
\end{align*}
where $x_i = D_i(x)$ and $y_i = D_i(y)$. Letting $\fa_n$ be the ideal generated
by the $(n+1)\times(n+1)$ minors of $M_n$, we have 
\begin{align*}
\fa_0 &= (x,y), 
\\
\fa_1 &= \fa_0^2 + (x y_1-y x_1),
\\
\fa_2 &= \fa_0\fa_1 + (y_2 x^2-y x_2 x-x_1 y_1 x+y x_1^2,\,
        x_2 y^2-x_1 y_1 y-x y_2 y+x y_1^2),
\\
\fa_3 &= \fa_0\fa_2 + \fa_1^2 +
\\   
& \hskip0.5cm
+ (
\begin{aligned}[t]
        & y_3 x^3-y x_3 x^2-x_2 y_1 x^2-x_1 y_2 x^2+2 y x_1 x_2 x+x_1^2 y_1 x-y x_1^3
        ,\\
        & y_1 y_2 x^2-y y_3 x^2-x_1 y_1^2 x+y^2 x_3 x+y^2 x_1 x_2+y x_1^2 y_1
        ,\\
        & x_2 y_1 y^2+x_1 y_2 y^2+x y_3 y^2-x_1 y_1^2 y-2 x y_1 y_2 y+x y_1^3-y^3 x_3
        ,\\
        & y_2^2 x^2
            -y_1 y_3 x^2
            +x_2 y_1^2 x
            +y x_3 y_1 x
            -2 y x_2 y_2 x
            -x_1 y_1 y_2 x \, +\\
         &   \hskip2cm +y x_1 y_3 x
            +y^2 x_2^2
            -y^2 x_1 x_3
            -y x_1 x_2 y_1
            +y x_1^2 y_2
    ).
\end{aligned}
\end{align*}
Notice that while the matrices $M_n$ remain simple and have an easily
recognizable structure, the corresponding ideals $\fa_n$ grow in complexity
quite fast.
\end{example}


\vfill

\end{document}